\documentclass[a4paper,12pt]{article}
\usepackage{amssymb, amsmath, afterpage}
\usepackage{amsfonts}
\usepackage[dvips]{graphics}
\newcommand{\mybox}{\hfill $\square$}
\newcommand{\+}[1]{\ensuremath{\overset{+}{#1}}}
\newcommand{\pli}[1]{\ensuremath{\overset{+}{\imath}}}
\newcommand{\mi}[1]{\ensuremath{\overset{-}{#1}}}
\newcommand{\arb}[1]{\ensuremath{\overset{\ast}{#1}}}

\newcommand{\sym}{\ensuremath{\mathrm{Sym}}}
\newcommand{\dih}{\ensuremath{\mathrm{Dih}}}
\newcommand{\I}{\ensuremath{\mathcal{I}}}
\newcommand{\J}{\ensuremath{\mathcal{J}}}
\newcommand{\ip}[2]{\ensuremath{\langle{#1},{#2}\rangle}} 
\newcommand{\eps}{\varepsilon}

\newcommand{\stab}{\ensuremath{\mathrm{Stab}}}
\newcommand{\supp}{\ensuremath{\mathrm{supp}}}
\newcommand{\qed}{\hfill $\square$ \\}

\newcommand{\rr}{\mathbb{R}}

\newtheorem{thm}{Theorem}[section]
\newtheorem{lemma}[thm]{Lemma}         
\newtheorem{prop}[thm]{Proposition}

\newtheorem{defn}[thm]{Definition}
\newtheorem{cor}[thm]{Corollary}
\newenvironment{proof}{\paragraph{Proof}}{\mybox}
\begin{document}
    \title{\Large\textbf{On Excess in Finite Coxeter Groups}}
\date{}
      \author{S.B. Hart and P.J. Rowley\thanks{The authors wish to acknowledge partial support for this work from
       the Department of Economics, Mathematics and Statistics at Birkbeck and the Manchester Institute for Mathematical Sciences (MIMS).
      }}
  \maketitle
\vspace*{-10mm}
\begin{abstract} \noindent For a finite Coxeter group $W$ and $w$ an
element of $W$ the \textit{excess} of $w$ is defined to be $e(w) =
\min\{\ell(x) + \ell(y) - \ell(w) \; | \; w=xy, \; x^2 = y^2 = 1\}$
where $\ell$ is the length function on $W$. Here we investigate the
behaviour of $e(w)$, and a related concept reflection excess, when
restricted to standard parabolic subgroups of $W$. Also the set of
involutions inverting $w$ is studied. (MSC2000: 20F55)
\end{abstract}

\section{Introduction}\label{intro}

This paper, continuing the investigations begun in \cite{invprod}
and \cite{zeroex}, studies further properties of excess in Coxeter
groups. First we recall the definition of excess.

Suppose $W$ is a Coxeter group with length function $\ell$, and set
\begin{center} $\mathcal{W} = \{w \in W \mid w = xy$ where $x, y, \in W$ and $x^2
= y^2 = 1\}$.\end{center}

 Then for $w \in \mathcal{W}$, the excess of $w$ is
$$e(w) = \min\{\ell(x) + \ell(y) - \ell(w) \; | \; w=xy, x^2 = y^2 =
1\}.$$

The main result in \cite{invprod} asserts that every element in
$\mathcal{W}$ is $W$-conjugate to an element whose excess is zero.
In a similar vein, \cite{zeroex} shows that if $W$ is a finite
Coxeter group, then every $W$-conjugacy class possesses at least one
element which simultaneously has minimal length in the conjugacy
class and excess equal to zero. The present paper explores other
properties of excess in finite Coxeter groups. So from now on we
assume $W$ is finite. Since every element in a finite Coxeter group
may be written (in possibly many ways) as $xy$ where $x^2 = y^2 =
1$, $e(w)$ is defined for all $w \in W$. Moreover, every element $w
\in W$ may be written as $xy$, where $x^2=y^2=1$ and $L(w) = L(x) +
L(y)$, where $L$ is the reflection length function on $W$. This
latter fact is (essentially) established in Carter \cite{Carter}
(see also Lemma 2.4 of \cite{zeroex}). This leads to the related
notion of reflection excess. For $w \in W$ its reflection excess
$E(w)$ is defined by
$$E(w) = \min\{\ell(x) + \ell(y) - \ell(w) \; | \; w=xy, x^2 = y^2 =
1, L(w) = L(x) + L(y)\}.$$
Clearly $E(w) \geq e(w)$. However $E(w)$ and $e(w)$ can be markedly
different -- see for example Proposition 3.3 of \cite{zeroex}.\\

The first issue we address here is how excess and reflection excess
behave on restriction to standard parabolic subgroups of $W$ -- as
is well-known, such subgroups are Coxeter groups in their own right.
If $W_J$ is a standard parabolic subgroup of $W$ and $w \in W_J$, we
let $e_J(w)$ (respectively $E_J(w)$) be the excess of $w$
(respectively reflection excess of $w$) considered as an element of
$W_J$.\pagebreak

Our main results are as follows.

\begin{thm} \label{parabolic} Let $W_J$ be a standard parabolic subgroup of $W$ and
let $w \in W_J$. Then $E_J(w) = E(w)$.\end{thm}

We remark that the proof of Theorem \ref{parabolic} is very short
and elementary, whereas its sister statement for excess requires a
lengthy case-by-case analysis. More than that there is a shock in
store as we now see.

\begin{thm} \label{parabolicexcess} Let $W_J$ be a standard parabolic subgroup of $W$ and
let $w \in W_J$. If W has no irreducible factors of type $D_n$, then
$e_J(w) = e(w)$.\end{thm}

The assumption that there be no direct factors of type $D_n$ in
Theorem \ref{parabolicexcess} cannot be omitted. In Section
\ref{StParaexcess} we give an example with $W$ of type $D_{12}$,
$W_J$ of type $D_{11}$ and an element $w$ of $W_J$ for which
$e_{J}(w) = 60$ but $e(w) = 46$. However there are a number of
positive results, to be found in Section \ref{StParaexcess}, for $W$
of type
$D_n$ provided we restrict $W_J$. \\

For $w \in W$, the set $\mathcal{I}_w$, which is defined as follows,

$$\I_w = \{x \in W \; | \; x^2 = 1, w^{x} = w^{-1}\}$$
is intimately connected with $e(w)$ and $E(w)$. This is because if
$x$ and $y$ are elements of $W$, with $x^2 = y^2 = 1$, such that $xy
= w$, then $w^x = (xy)^x = yx = w^{-1}$ and similarly $w^y =
w^{-1}$. Therefore $x,
y \in \mathcal{I}_w$.  Thus $\mathcal{I}_w$ is always non-empty.\\

%
%
%


For $X \subseteq W$, we define in Section 2 a certain subset $N(X)$
of the positive roots of $W$. The Coxeter length (or just length) of
$X$, denoted by $\ell(X)$, is defined to be $|N(X)|$ (see
\cite{coxlen}). A consequence of our next theorem is that for all $w
\in W$, $\ell(w) \leq \ell(\mathcal{I}_w)$.

\begin{thm}\label{nwni} For all $w \in W$, $N(w) \subseteq
N(\mathcal{I}_w)$. \end{thm}

This paper is arranged as follows. Our next section gathers together
relevant background material while reviewing much of the standard
notation used for Coxeter groups. Theorems \ref{parabolic} and
\ref{parabolicexcess} are established in Section \ref{StParaexcess};
the former, being an easy consequence of Lemma \ref{basics}(ii), is
proved first. Then Lemma \ref{afterthought} gives criteria for
recognizing when two involutions fail to be a spartan pair (see
Definition \ref{2.3} for the definition of a spartan pair). With
this result to hand we then prove the pivotal Propositions
\ref{supp} and \ref{suppdn}. In fact Propositions \ref{supp} and
\ref{suppdn} mark the parting of the ways for type $A$, $B$ and type
$D$. Corollary \ref{overlap} and Proposition \ref{swapcycle} combine
to pin down the $2$-cycles $(\arb a\; \arb b)$ of spartan pairs.
With this information we are then able to complete, in Theorems
\ref{anexcess} and \ref{bnexcess}, the proof of Theorem
\ref{parabolicexcess}. All is not lost for type $D$, as Theorem
\ref{dnexcess} demonstrates, with
various conditions which guarantee that $e_J(w) = e(w)$.\\

Our final section investigates $N(\mathcal{I}_w)$ for $w \in W$.
Proposition \ref{cusplen} and Lemma \ref{centre} reveal that, under
certain circumstances, $N(\mathcal{I}_w) = \Phi^+$ (though this is
not always the case) and the balance of this section presents a
proof of Theorem \ref{nwni}.

\section{Background Results and Notation}\label{back}

We briefly recall the standard notation used for finite Coxeter
groups $W$ and their root systems. To begin with, by definition, $W$
has a presentation of the form

\[ W = \langle R \; | \; (rs)^{m_{rs}} = 1, r, s \in R\rangle \]
where $m_{rs} = m_{sr} \in \mathbb{N}$, $m_{rr} = 1$ and $m_{rs}
\geq 2$ for $r, s \in R, r\neq s$. We put $R = \{r_1, \dots ,r_n \}$
-- the $r_i$ are called the fundamental reflections of $W$. The
length of an element $w$ of $W$, denoted by $\ell(w)$, is defined to
be
\[\ell(w) = \left\{ \begin{array}{l} \min\{ l \; | \; w=r_{i_1}\cdots r_{i_l}, r_{i_j} \in R\} {\mbox{ if $w \neq 1$}} \\
0 {\mbox{ if $w=1$}.} \end{array}\right.\]

Taking $V$ to be a real euclidean vector space with basis $\Pi = \{
\alpha_r \; | \; r \in R\}$ and norm $|| \;\; ||$, we define a
symmetric bilinear form $\ip{ \; }{ \; }$ on $V$ by

\[\ip{\alpha_r}{\alpha_s} =  - ||\alpha_r|| \; ||\alpha_s|| \cos\left(\frac{\pi}{m_{rs}}\right)
, (r, s \in R).\]

Now for $r, s\in R$ we define
\[ \alpha_s\cdot r = \alpha_s - 2\frac{\ip{\alpha_r}{\alpha_s}}{\ip{\alpha_r}{\alpha_r}}\alpha_r,\]
which extends to an action of $W$ on $V$. This action is faithful
and respects $\ip{ \; }{ \; }$ (see \cite{humphreys}). We remark
that traditionally the action of a Coxeter group on its root system
is on the left, but since in this paper we will largely be working
with permutation groups, which usually act on the right, we have
chosen to act on the right throughout. The following subset of $V$
\[\Phi = \{\alpha_r \cdot w\; | \; r \in R, w \in W\}\]
is the root system of $W$. Setting $\Phi^+ = \{\sum_{r\in
R}\lambda_r\alpha_r \in \Phi \; | \; \lambda_r \geq 0 {\mbox{ for
all $r$}}\}$ and $\Phi^- = -\Phi^+$ we have the fundamental fact
that $\Phi$ is the disjoint union $\Phi^+ \dot\cup \Phi^-$ (see
\cite{humphreys} again), the sets $\Phi^+$ and $\Phi^-$ being
referred to, respectively, as the positive and negative roots of
$\Phi$. Let $\alpha$ be a positive root. Then $\alpha = \alpha_r
\cdot w$ for some $w \in W$ and $r \in R$. Define $r_{\alpha} =
w^{-1}rw$. Then $\alpha \cdot r_{\alpha} = -\alpha$. Such an element
as $r_{\alpha}$ is called a {\em reflection} of $W$.

For $X$ a subset of $W$ we define
\[N(X) = \{ \alpha \in \Phi^+ \; | \; \alpha\cdot w \in \Phi^- {\mbox{ for some $w\in X$}}\}.\]

If $X = \{w\}$, we write $N(w)$ instead of $N(\{w\})$. Clearly $N(X)
= \cup_{w \in X}N(w)$. The Coxeter length of $X$, $\ell(X)$, is
defined to be $\ell(X) = |N(X)|$ -- for more on the Coxeter length
of subsets of Coxeter groups, see \cite{coxlen}. The connection
between $\ell(w)$ and the root system of $W$ is contained in our
next lemma.

\begin{lemma}\label{millie}Let $w \in W$ and $\alpha \in \Phi^+$.
\begin{trivlist}
\item (i) If $\ell(r_{\alpha}w) > \ell(w)$ then $ \alpha\cdot w \in
\Phi^+$ and if $\ell(r_{\alpha}w) < \ell(w)$ then $\alpha \cdot w\in
\Phi^-$. In particular, $\ell(r_{\alpha}w) < \ell(w)$ if and only if
$\alpha \in N(w)$. \item (ii) $\ell(w) = |N(w)|$.
\end{trivlist}
\end{lemma}

\begin{proof} Parts (i) and (ii) are, respectively, Propositions 5.7 and 5.6 of \cite{humphreys}.
\end{proof}

\begin{lemma} \label{lengthadd} Let $g, h \in W$. Then
$$N(gh) = N(g)\setminus [-N(h)\cdot g^{-1}] \cup [N(h)\setminus
N(g^{-1})]\cdot g^{-1}.$$
Hence $\ell(gh) = \ell(g) + \ell(h) -2\mid N(g^{-1})\cap
N(h)\mid$.\end{lemma}

\begin{proof} See Lemma 2.2 of \cite{invprod}.
\end{proof}\\

For $J$ a subset of $R$ define $W_J = \langle J \rangle$. Such a
subgroup of $W$ is referred to as a standard parabolic subgroup.
Standard parabolic subgroups are Coxeter groups in their own right
with root system
\[\Phi_J = \{\alpha_{r}\cdot w \; | r \in J, w \in W_J\}\]
(see Section 5.5 of \cite{humphreys} for more on this). A conjugate
of a standard parabolic subgroup is called a parabolic subgroup of
$W$, and a cuspidal element of $W$ is an element not contained in
any proper parabolic subgroup of $W$.

\begin{defn} \label{2.3} Let $w \in W$. We call $(x,y)$ a spartan pair for $w$
if $x, y \in W$, $x^2 = y^2 = 1$, $w = xy$ and $\ell(x) + \ell(y) -
\ell(w) = e(w)$.\end{defn}%
 A consequence of Lemma \ref{lengthadd} is
that if $x, y \in W$ with $x^2=y^2=1$ and $w=xy$, then $(x,y)$ is a
spartan pair for $w$ if and only if $2|N(x)\cap N(y)| = e(w)$.
Letting $V_{\lambda}$ denote the $\lambda$-eigenspace of $V$
($\lambda \in \rr$) we introduce the following subset $\J_w$ of
$\I_w$, $w\in W$.

$$\J_w = \{x \in W \; | \; x^2 = 1, w^{x} = w^{-1}, V_1(w) \subseteq
V_1(x)\}.$$

\begin{lemma}\label{basics} Suppose that  $w \in W$. Then
\begin{trivlist}
\item{(i)}  $e(w)$ is the sum
of the excesses and $E(w)$ is the sum of the reflection excesses of
the projections of $w$ into
the irreducible direct factors of $W$; and%

\item{(ii)} $\J_w$ is the set of $x$ such that $w = xy$ where $x^2 =
y^2 =1$ and $L(w) = L(x) + L(y)$.\end{trivlist}\end{lemma}

\begin{proof} Since $\ell(w)$, respectively $L(w)$, is the sum of
the lengths, respectively reflection lengths, of the projections of
$w$ into the irreducible direct factors of $W$, (i) follows easily.
For (ii), see Lemma 3.2(i) of \cite{zeroex}.\end{proof}\\

In view of Lemma \ref{basics}(i), irreducible finite Coxeter groups
appear frequently in our proofs. Such groups have been classified by
Coxeter \cite{coxeter} (see also \cite{humphreys}).

\begin{thm}\label{classification} An irreducible finite Coxeter group is either of type
$A_n (n\geq 1)$, $B_n (n \geq 2)$, $D_n (n \geq 4)$, $\dih(2m)$ (a
dihedral group of order $2m$, $m \geq 5$), $E_6$, $E_7$, $E_8$,
$F_4$, $H_3$ or $H_4$.
\end{thm}

We shall employ the following explicit descriptions of the Coxeter
groups of types $A_n, B_n$ and $D_n$ and their root systems. First,
$W(A_n)$ may be viewed as being $\sym(n+1)$ with the set of
fundamental reflections given by $\{(12), (23), \dots , (n \; n+1)
\}$, while elements of $W(B_n)$ can be thought of as signed
permutations of $\sym(n)$. A cycle in an element of $W(B_n)$ is of
negative sign type if it has an odd number of minus signs, and
positive sign type otherwise. We take
$\{(\overset{+}{1}\overset{+}{2}), (\overset{+}{2}\overset{+}{3}),
\ldots, (\overset{+}{n-1} \; \overset{+}{n}), (\overset{-}{n})\}$ to
be the fundamental reflections in $W(B_n)$. An element $w$ expressed
as a product $g_1g_2\cdots g_k$ of disjoint signed cycles is {\em
positive} if the product of all the sign types of the cycles is
positive, and negative otherwise. The group $W(D_n)$ consists of all
positive elements of $W(B_n)$ and we take the fundamental
reflections of $W(D_n)$ to be $r_1 =
(\overset{+}{1}\overset{+}{2})$, $r_2 =
(\overset{+}{2}\overset{+}{3})$, $\ldots$, $r_{n-1} =
(\overset{+}{n-1} \; \overset{+}{n})$, $r_n =  (\overset{-}{n-1}\;
\overset{-}{n})$. Even if $w$ is positive, it may contain negative
cycles, which we wish on occasion to consider separately, so when
considering elements of
$W(D_n)$ we sometimes work in $W(B_n)$.\\

Let $\{e_i\}$ be an orthonormal basis with respect to the form
$\ip{\;}{\;}$ for $V$. For $\sigma \in W(A_n)$ define $e_i\cdot
\sigma = e_{i\sigma}$ -- note that our permutations and signed
permutations will always act on the right. The roots for $W(A_n)$
are $\pm (e_i - e_j)$ for $1\leq i < j\leq n$, with the positive
roots being $\{e_i - e_j \; | \; 1 \leq i < j \leq n\}$. The
positive roots of $W(B_n)$ are of the form $e_i \pm e_j$ for $1 \leq
i < j \leq n$ and $e_i$ for $1\leq i \leq n$. The positive roots of
$W(D_n)$ are of the form $e_i \pm e_j$ for $1 \leq i < j \leq n$.

\section{Excess and Standard Parabolic Subgroups}
\label{StParaexcess}

The main aim of this section is to prove Theorems \ref{parabolic}
and \ref{parabolicexcess}. So let $J$ be a subset of $R$.

\paragraph{Proof of Theorem \ref{parabolic}} Let $w \in W_J$ and $x, y \in W$ with $x^2 = y^2 = 1$ and
$xy=1$. Then, by Lemma \ref{basics}(ii), $x \in \J_w$. Therefore
$V_1(w) \subseteq V_1(x)$. Let $U = \{v \in V\; |\; W_J \subseteq
\stab(v)\}$. Then for all $u \in U$, $u \in V_1(w) \subseteq
V_1(x)$. Hence $U \subseteq V_1(x)$ and so $x \in W_J$. Thus
$\J_w|_{W_J} = \J_w$. Hence
$E_J(w) = E(w)$.\qed\\


We direct our attention to Theorem \ref{parabolicexcess} -- first we
must establish a number of preliminary results about spartan pairs.

\begin{lemma} \label{afterthought} Suppose $x$ and $y$ are involutions in $W$. If $z$
is an involution centralizing both $x$ and $y$, such that $\ell(zx)
< \ell(x)$ and $\ell(zy)< \ell(y)$, then $|N(zx) \cap N(zy)| < |N(x)
\cap N(y)|$. Hence $(x, y)$ is not a spartan pair for
$w=xy$.\end{lemma}

\begin{proof} By Lemma \ref{lengthadd} and the observation that
for an involution $\sigma$, $N(\sigma) = -\sigma N(\sigma)$, we
obtain
\begin{eqnarray*} N(zx) &=& N(z)\setminus \left(-N(x)\cdot z\right) \;
\dot\cup \;
\left(N(x)\setminus N(z)\right)\cdot z\\
&=& \left[ -\left(N(z)\setminus N(x)\right) \;\dot\cup \;
\left(N(x)\setminus N(z)\right)\right] \cdot z.\end{eqnarray*}%
 Similarly
\[N(zy) = \left[-\left(N(z)\setminus N(y)\right)\; \dot\cup \;
\left(N(y)\setminus N(z)\right)\right] \cdot z.\]
Notice that $\ell(zx) = \ell(z) + \ell(x) - 2|N(z) \cap N(x)|$.
Hence $\frac{1}{2}|N(z)| - |N(z) \cap N(x)| = \frac{1}{2}(\ell(zx) -
\ell(x))$, and the same is true for $y$. Therefore
\begin{eqnarray*} |N(zx) \cap N(zy)| &=& |N(z)\setminus \left(N(x)
\cup N(y)\right)| + |\left(N(x)\cap N(y)\right)\setminus N(z)|\\
&=& |N(z)| - |N(z)\cap N(x)| - |N(z) \cap N(y)| + |N(x)\cap
N(y)|\\
&=& |N(x) \cap N(y)| -\textstyle\frac{1}{2}\left(\ell(y) -
\ell(zy)\right) -\textstyle\frac{1}{2}\left(\ell(x) - \ell(zx)\right)\\
&<& |N(x) \cap N(y)|.\end{eqnarray*} %
If $(x, y)$ were a spartan pair for $w=xy$, then $e(w) = 2|N(x)\cap
N(y)|$. But $zx, zy \in \I_w$ with $(zx)(zy) = w$ and $2|N(zx) \cap
N(zy)| < e(w)$, a contradiction. Therefore $(x, y)$ cannot be a
spartan pair for $w$.\end{proof}\\

For the rest of this section, $W_n$ is a Coxeter group of type
$A_{n-1}$, $B_n$ or $D_n$; the elements of $W_n$ are therefore
cycles or signed cycles of $\sym(n)$. The notation $W = W(n_1,
\ldots, n_k)$ means $W$ is of type $W_k$ with support $\{n_1,
\ldots, n_k\}$. Suppose that $W_J$ is a maximal parabolic subgroup
of $W_n$. Then for some $m$ with $1 \leq m \leq n$, we may assume
that $W_J$ is of the form $\sym(1, 2, \ldots, m) \times W(m+1,
\ldots, n)$. Note that the case $m=n$ is not included if $W$ is of
type $A_{n-1}$. If $W$ is of type $D_n$, the length preserving graph
automorphism means that it is not necessary to consider separately
the case $W_J = \langle (\overset{+}{1}\overset{+}{2}),
(\overset{+}{2}\overset{+}{3}), \ldots, (\overset{+}{n-2} \;
\overset{+}{n-1}), (\overset{-}{n-1}\; \overset{-}{n})\rangle$,
as this will be covered by the case $m=n$. We will abuse notation
and deem $D(n_1, n_2)$ and $D(n_1,n_2,n_3)$ to be of types $A_1 \times A_1$ and $A_3$ respectively.\\

We remark that involutions in $W$ only contain cycles of the form
$(\+ a \; \+ b), (\mi a\; \mi b), (\+ a)$ and $(\mi a)$. That is,
1-cycles and positive 2-cycles.\\

For $u \in W$, the {\em positive support} of $u$, denoted
$\supp^+(u)$, is the set of all $a \in \{1, \ldots, n\}$ for which
$e_a \cdot u \neq e_a$. So $\supp^+(u)$, in the case of type $B$ and
$D$, differs from the support of $u$ as a permutation of the set
$\{\pm 1, \dots , \pm n \}$ by only considering the elements of
$\{1, \dots , n \}$ which are moved by $u$.

\begin{prop} \label{supp} Suppose $W_n$ is of type $A_{n-1}$ or $B_n$ and let $w \in W_n$.
If $(x, y)$ is a spartan pair for $w$, then $\supp^+(x) \cup \supp^+(y) \subseteq
\supp^+(w)$.\end{prop}
\begin{proof} Suppose for a contradiction that there exists $i \in \supp^+(y) \setminus \supp^+(w)$.
Then $ e_i\cdot y = \pm e_j$ for some $j$ with either $i\neq j$ or
$e_i\cdot y = -e_i$. Now $ e_i \cdot w = e_i$ forces $e_i\cdot x =
e_i\cdot y$. Define a positive root $\alpha$ as follows:
$$\alpha = \left\{\begin{array}{ll}e_i - e_j & {\mbox{ if $e_i\cdot y = e_j$, $j>i$}};
\\
e_j - e_i & {\mbox{ if $e_i \cdot y= e_j$, $j<i$}};
\\
e_i+e_j & {\mbox{ if $e_i\cdot y = -e_j, j\neq i$; and}}\\
e_i & {\mbox{ if $e_i\cdot y = -e_i$}}.\end{array}\right.$$ Then
$\alpha\cdot x = \alpha\cdot y = -\alpha$. This means that
$r_{\alpha}$ centralizes both $x$ and $y$. Moreover
$\ell(r_{\alpha}x) < \ell(x)$ and $\ell(r_{\alpha}y) < \ell(y)$. By
Lemma \ref{afterthought} this contradicts the fact that $(x,y)$ is a
spartan pair. Hence $\supp^+(y) \subseteq \supp^+(w)$. The same
argument with $x$ and $w^{-1}$ implies that $\supp^+(x) \subseteq
\supp^+(w^{-1}) = \supp^+(w)$. Therefore $\supp^+(x) \cup \supp^+(y)
\subseteq \supp^+(w)$.
\end{proof}\\

\begin{prop} \label{suppdn} Suppose $W_n$ is of type $D_n$ and $w
\in W_n$. If $(x, y)$ is a spartan pair for $w$, then $|\supp^+(y)
\setminus \supp^+(w)| \leq 1$ and if $i \in \supp^+(y)\setminus
\supp^+(w)$ then $e_i\cdot y = e_i\cdot x = -e_i$. Furthermore
$\supp^+(y) \setminus
\supp^+(w) = \supp^+(x) \setminus \supp^+(w)$.\end{prop} %
\begin{proof}
Suppose $i \in \supp^+(y)\setminus \supp^+(w)$ is such that
$e_i\cdot y = \pm e_j$ for some $j \neq i$. Then $e_i\cdot x =
e_i\cdot y$ and we define the positive root $\alpha$ as in the proof
of Proposition \ref{supp}, noting that the possibility $\alpha =
e_i$ does not occur, and so $\alpha$ is indeed a root of $D_n$.
Again, $\ell(r_{\alpha}x) < \ell(x)$ and $\ell(r_{\alpha}y) <
\ell(y)$, contradicting the fact that $(x, y)$ is a spartan pair.\\
Therefore, $\supp^+(y) \setminus \supp^+(w) \subseteq \{i\; |\;
 e_i\cdot y = -e_i\}$. Suppose $\{i, k\} \subseteq \supp^+(y)
\setminus \supp^+(w)$ with $i\neq k$. Let $\beta = e_i + e_k \in
\Phi^+$. Then $\beta\cdot y = \beta\cdot x = -\beta$ and hence
$\ell(r_{\beta}x) < \ell(x)$ and $\ell(r_{\beta}y) < \ell(y)$,
contradicting the fact that $(x,y)$ is a spartan pair. Hence
$\supp^+(y) \setminus \supp^+(w)$ contains at most one element $i$,
and $e_i\cdot y = -e_i$. Since $e_i\cdot x = e_i\cdot y$, we have
$\supp^+(y) \setminus \supp^+(w) \subseteq \supp^+(x) \setminus
\supp^+(w)$. Repeating the argument with $x$ and $w^{-1}$ gives the
reverse inclusion, forcing $\supp^+(y) \setminus \supp^+(w) =
\supp^+(x) \setminus
\supp^+(w)$.\end{proof}\\

Note that there are examples of spartan pairs $(x,y)$ for $w$ where
$\supp^+(y)$ is not contained in $\supp^+(w)$. These examples are
the source of (infinitely many) cases  in which $e_J(w) > e(w)$. One
such is the following: $w = (\+{2} \; \+{4} \; \+{6} \; \+{8} \;
\+{10} \; \mi{12}\; \+{11} \; \+{9} \; \+{7} \; \+{5} \; \mi{3}) \in
D_{12}$. As a product of fundamental reflections $w =
[468.10.3456789.10.11.12.10.987654323579.11]$ where $10$ is the
branch node of the $D_{12}$ diagram, and in this expression for $w$
we have written $i$ instead of $r_i$. Clearly $w$ lies in a standard
parabolic subgroup $W_J$ of type $D_{11}$. It can be shown that
$e_J(w) = 60$, whereas $e(w) = 46$, given by the spartan pair $(x,
y)$ where $x = (\mi{1})(\+{2} \; \+{3})(\+ {4} \; \+{5})(\+{6} \;
\+{7})(\+{8} \; \+{9})(\+{10} \; \+{11})(\mi{12})$, $y =
(\mi{1})(\mi{2})(\+{3} \; \+{4})(\+{5} \; \+{6})(\+{7} \;
\+{8})(\+{9} \; \+{10})(\+{11} \; \+{12})$.

\begin{prop} \label{swapcycle} Suppose $(x,y)$ is a spartan pair for $w \in W_n$.
If $(\+{a_1} \cdots \+{a_k})$ and $(\overset{\eps_{1}}{b_1} \cdots
\overset{\eps_{k}}{b_k})$ are disjoint $w$-cycles for which
$(\+{a_1} \cdots \+{a_k})^y = (\overset{\eps_{1}}{b_1} \cdots
\overset{\eps_{k}}{b_k})^{-1}$, then $\max \{a_i\} > \min\{b_i\}$.
\end{prop}

\begin{proof} Without loss of generality, assume $a_1y = \pm b_k$,
$\ldots$, $a_iy = \pm b_{k+1-i}$, $\ldots$, $a_ky = \pm b_1$. Let $T
= \{1, \ldots, n\}\setminus\{a_1, \ldots, a_k, b_1, \ldots, b_k\}$.
Then $y = y_1y_2$ for some involution $y_1$ with $\supp(y_1)
\subseteq T$ and $y_2 = (\overset{\rho_1}{a_1} \;
\overset{\rho_1}{b_k}) \cdots (\overset{\rho_k}{a_k} \;
\overset{\rho_k}{b_1})$. Define $z = (\overset{\rho_k}{a_1} \;
\overset{\rho_k}{b_1})(\overset{\rho_{k-1}}{a_2} \;
\overset{\rho_{k-1}}{b_2}) \cdots (\overset{\rho_1}{a_k} \; \overset{\rho_1}{b_k})$. Now%
\begin{eqnarray*} yz &=& y_1y_2z = y_1
\prod_{i=1}^{k}(\overset{\rho_i}{a_i} \;
\overset{\rho_i}{b_{k+1-i}})%
\prod_{i=1}^{k}(\overset{\rho_{k+1-i}}{a_i} \;
\overset{\rho_{k+1-i}}{b_{i}})\\%
&=& y_1 \prod_{i=1}^{\lfloor
k/2\rfloor}(\overset{\rho_i\rho_i}{a_i} \;
\overset{\rho_i\rho_i}{a_{k+1-i}})%
\prod_{i=1}^{\lfloor k/2\rfloor}(\overset{\rho_i\rho_{k+1-i}}{b_i}
\;
\overset{\rho_i\rho_{k+1-i}}{b_{k+1-i}})\\%
&=& y_1 \prod_{i=1}^{\lfloor k/2\rfloor}(\+{a_i} \;
\+{a_{k+1-i}})%
\prod_{i=1}^{\lfloor k/2\rfloor}(\overset{\rho_i\rho_{k+1-i}}{b_i}
\; \overset{\rho_i\rho_{k+1-i}}{b_{k+1-i}}).\end{eqnarray*} %
Therefore  $yz$ is an involution. Next we show that $xz$ is an
involution. We know that $w^y = w^{-1}$. Hence for $1 < i \leq k$,
\begin{eqnarray*}e_{a_{i-1}} &=& e_{a_i}\cdot ywy \\ &=&
\rho_ie_{b_{k+1-i}}\cdot wy \\ &=&
\eps_{k+1-i}\rho_ie_{b_{k+2-i}}\cdot y
\\ &=& \rho_{i-1}\rho_i\eps_{k+1-i}e_{a_{i-1}}.\end{eqnarray*}%
Therefore $\rho_{i-1}\rho_i = \eps_{k+1-i}$. Similarly $\rho_k\rho_1
= \eps_{k}$. This allows us to calculate $e_{a_j}\cdot zwz$ and
$e_{b_j}\cdot zwz$ for $1\leq j \leq k$:
\begin{eqnarray*} e_{a_j}\cdot zwz &=&
\rho_{k+1-j}\eps_j\rho_{k-j} e_{a_{j+1}} = e_{a_{j+1}} =
e_{a_j}\cdot w\\%
e_{b_j}\cdot zwz &=& \rho_{k+1-j}\rho_{k-j} e_{a_{j+1}} =
\eps_{j}e_{b_{j+1}} = e_{b_j}\cdot w\end{eqnarray*} %
Hence $zwz = w$ and so $z$ centralizes $x = wy$. \smallskip\\
Suppose for a contradiction that $\max\{a_i\} < \min\{b_i\}$. We
will show that $\ell(zy) < \ell(y)$ and $\ell(zx) < \ell(x)$. Write
$z_i = \prod_{t=1}^i (\overset{\rho_{k+1-t}}{a_t}
\; \overset{\rho_{k+1-t}}{b_t})$. Now%
\begin{eqnarray*} (e_{a_{i+1}}- \rho_{k-i}e_{b_{i+1}})\cdot z_iy
&=& (e_{a_{i+1}}- \rho_{k-i}e_{b_{i+1}})\cdot y  = (e_{a_{i+1}}\cdot
\prod_{i=1}^{k}(\overset{\rho_i}{a_i} \;
\overset{\rho_i}{b_{k+1-i}})- \rho_{k-i}e_{b_{i+1}})\\
&=& \rho_{i+1}e_{b_{k-i}} -\rho_{k-i}\rho_{k-i}e_{a_{k-i}} =
-e_{a_{k-i}} \pm e_{b_{k-i}}  \in \Phi^{-}.
\end{eqnarray*}
Hence by Lemma \ref{millie} $\ell(zy) = \ell(z_ky) < \ell(z_{k-1}y)
< \cdots < \ell(z_1y) < \ell(y)$. Similarly
$$(e_{a_{i+1}} - \rho_{k-i}e_{b_{i+1}})\cdot z_ix =
(-e_{a_{k-i}} \pm e_{b_{k-i}})\cdot w^{-1}= -e_{a_{k-i-1}} \pm
e_{b_{k-i-1}} \in \Phi^{-}$$ and so $\ell(zx) < \ell(x)$. But by
Lemma \ref{afterthought} this contradicts the fact that $(x, y)$ is
a spartan pair. Therefore $\max \{a_i\} > \min\{b_i\}$.
\end{proof}\\

\begin{cor}\label{overlap} Suppose $w \in W_n$ is contained in a
maximal parabolic subgroup $W_J$ of $W_n$ which is of the form
$\sym(1, 2, \ldots, m) \times W(m+1, \ldots, n)$. If $(x, y)$ is a
spartan pair for $w$, then for every 2-cycle $(\arb a\; \arb b)$ of
$x$ or $y$, either $\{a, b\} \subseteq \{1, \ldots, m\}$ or $\{a,
b\} \subseteq \{m+1, \ldots, n\}$.
\end{cor}

\begin{proof}
Since $x = wy$, it is enough to prove the result for 2-cycles $(\arb
a\; \arb b)$ of $y$. %
Without loss of generality, assume $a < b$. Suppose that $\{a, b\}
\not\subseteq \{m+1, \ldots, n\}$. Then $a \in \{1, \ldots, m\}$. If
$b$ lies in the same $w$-cycle as $a$, then $b$ is forced to lie in
$\{1, \ldots, m\}$ (since $w \in W_J$) and we are done. If $b$ lies
in a different $w$-cycle then since $w^y = w^{-1}$, the $w$-cycles
containing $a$ and $b$ respectively are of the form $(\+{a_1} \cdots
\+{a_k})$ and $(\overset{\eps_{1}}{b_1} \cdots
\overset{\eps_{k}}{b_k})$. By Proposition \ref{swapcycle}, we see
that $\max \{a_i\} > \min\{b_i\}$. Since $\max\{a_i\} \leq m$, we
get $\{b_1, \ldots, b_k\} \subseteq \{1, \ldots, m\}$. In
particular, $\{a, b\} \subseteq \{1, \ldots, m\}$.\end{proof}\\

\begin{thm} \label{anexcess} Suppose $W = W(A_n)$ and $w \in W_J \leq W$. Then
$e_J(w) = e(w)$. \end{thm}
\begin{proof} We may assume that $W_J$ is maximal and hence of the form $\sym(1, 2, \ldots, m) \times
\sym(m+1, \ldots, n)$ for $1\leq m \leq n-1$. Let $(x, y) $ be a
spartan pair for $w$. Note that since $x = wy$, $x$ and $y$ lie in
the same right $W_J$-coset. It is therefore enough to show that $y
\in W_J$. By Corollary \ref{overlap}, for every 2-cycle $(a\; b)$ of
$y$, either $\{a, b\} \subseteq \{1, \ldots, m\}$ or $\{a, b\}
\subseteq \{m+1, \ldots, n\}$. Hence $y$ is a product of commuting
reflections each of which lies in $W_J$, forcing $y \in W_J$. Thus
every spartan pair for $w$ is already to be found in $W_J$ and so
$e_J(w) = e(w)$.
\end{proof}\\

\begin{thm} \label{bnexcess} Suppose $W = W(B_n)$ and $w \in W_J \leq W$. Then
$e_J(w) = e(w)$. \end{thm}
\begin{proof} We may assume that $W_J$ is maximal and hence of the form $\sym(1, 2, \ldots, m) \times
B(m+1, \ldots, n)$ for $1\leq m \leq n$. Let $(x, y)$ be a spartan
pair for $w$. Again we will show that $y \in W_J$. By Corollary
\ref{overlap}, for every 2-cycle $(\arb a\; \arb b)$ of $y$, either
$\{a, b\} \subseteq \{1, \ldots, m\}$ or $\{a, b\} \subseteq \{m+1,
\ldots, n\}$. Let $S = \{a \in \{1, \ldots, m\}\; |\; e_a \cdot y\in
\Phi^-\}$. Assume for a contradiction that $S$ is nonempty. Now
define $z = \prod_{a \in S} (\mi a)$. Repeated applications of Lemma
\ref{millie} show that $\ell(zy) < \ell(y)$ and furthermore
$\ell(zx) < \ell(z)$ (since $e_a\cdot x = e_a\cdot yw^{-1}$ and
$e_a\cdot w^{-1} \in \Phi^+$ for all $a \in \{1, \ldots, m\}$). Note
also that $z$ centralizes $y$. To show that $z$ centralizes $x$,
suppose $a \in S$. Then $e_a\cdot y = -e_b$ for some $b \in \{1,
\ldots, m\}$. So $(e_a\cdot w)\cdot y = e_a \cdot yw^{-1}= -e_b\cdot
w^{-1} \in \Phi^-$. Therefore $a \in S$ implies $aw \in S$ and hence
for any $w$-cycle $(\+ a_1 \cdots \+ a_k)$ with $\{a_1, \ldots,
a_k\} \subseteq \{1, \ldots, m\}$, either $e_{a_i}\cdot z = e_{a_i}$
for all $1\leq i \leq k$ or $e_{a_i}\cdot z = -e_{a_i}$ for all $1
\leq i \leq k$. Consequently $z$ centralizes $w$. We deduce that $z$
centralizes $x$ as well as $y$, and so by Lemma \ref{afterthought},
$(x, y)$ cannot be a spartan pair, a contradiction. Thus $S$ is
empty and $e_a\cdot y \in \Phi^+$ for all $a \in \{1, \ldots, m\}$.
Taken with the fact that the 2-cycles of $y$ do not interchange
elements of $\{1, \ldots, m\}$ and $\{m+1,\ldots, n\}$, this shows
that $y \in W_J$. Hence $x = wy \in W_J$. Thus every spartan pair
for $w$ is contained in $W_J$, which implies $e_J(w) = e(w)$.
\end{proof}\\

\paragraph{Proof of Theorem \ref{parabolicexcess}} By
Lemma \ref{basics}(i), it is enough to prove the result for $W$
an\linebreak  irreducible finite Coxeter group not of type $D_n$.
Types $A_n$ and $B_n$ have been dealt with in Theorems
\ref{anexcess} and \ref{bnexcess}. Theorem \ref{parabolicexcess}
trivially holds for dihedral groups. Types $E_6$, $E_7$, $E_8$,
$F_4$, $H_3$ and $H_4$ have been checked using the computer algebra
package {\sc Magma}\cite{magma}. We discuss the details of these
calculations using $W=W(E_8)$ as an example. If for each maximal
parabolic subgroup $W_J$ of $W$ we know that for all $K \subset J$
and for all $w \in W_K$, $e_J(w) = e_K(w)$, then it suffices to
verify Theorem \ref{parabolicexcess} for all the maximal parabolic
subgroups of $W$ -- that is, that $e_J(w) = e(w)$ for all maximal
parabolic subgroup $W_J$ of $W$. So, for example, if $W_J$ is of
type $A_n$, then we may apply Theorem \ref{anexcess} for $W_K
\subseteq W_J$. However we must beware of standard parabolic
subgroups of $W$ of type $D_n$ -- these must be checked directly.
Among the standard parabolic subgroups of $W$, the most challenging
calculation occurs when $W_J$ is of type $E_7$. Set $E = \{y \mid y
\in W_J$ and $y$ has order greater than $2 \}$. We note that $|E| =
2,892,832$. We need to check that $e(y) = e_J(y)$ for all $y \in E$
(this is because we know that $e(w) = 0 = e_J(w)$ for all $ w \in
W_J \backslash E$). Below we give the {\sc Magma} code which was
used for the calculations in the groups of types $E_6, E_7, E_8,
F_4, H_3$ and $H_4$ (with \verb"ans = {0}"
 being the output obtained in
all cases).  \\

In the routine $H$ denotes the standard parabolic subgroup of the Coxeter group $W$.

\begin{verbatim}
E:={y: y in H |Order(y) gt 2};
ans:={ };for x in E do NH:=Normalizer(H,sub<W|x>);
CW:=Centralizer(W,x);
CH:=Centralizer(H,x);
SH:=Sylow(NH,2);
RH:=sub<W|SH,CH>;
TT:=Transversal(RH,CH);
for i:=1 to #TT do
if x^TT[i] eq x^-1 then inverter:=TT[i];end if;
end for;
CosH:={c*inverter : c in CH};CosW:={c*inverter : c in CW};
tempW:={y: y in CosW |Order(y) eq 2};
tempH:={y: y in CosH |Order(y) eq 2};
lenx:=CoxeterLength(W,x);
minW:=Min({CoxeterLength(W,x*y)+ CoxeterLength(W,y) - lenx : y in tempW});
minH:=Min({CoxeterLength(W,x*y)+ CoxeterLength(W,y) - lenx : y in tempH});
ans:=ans join {minH - minW};end for;
\end{verbatim}

The final result in this section shows that excess does behave well in type $D_n$ with
respect to certain parabolic subgroups and cycle types of elements.

\begin{thm} \label{dnexcess} Suppose $W = W(D_n)$ and $w \in W_J \leq W$, where
$W_J$ is of the form \linebreak $\sym(1, 2, \ldots, m) \times D(m+1,
\ldots, n)$. Write $w = w_1w_2$, where $w_1 \in \sym(1, 2, \ldots,
m)$ and $w_2 \in D(m+1, \ldots, n)$. If either $m=n$, or $w_2$
contains a 1-cycle, or $w_2$ consists only of even, positive cycles,
then $e_J(w) = e(w)$.
\end{thm}
\begin{proof} Let $(x, y)$ be a spartan
pair for $w$. We will show that $y \in W_J$. By Corollary
\ref{overlap}, for every 2-cycle $(\arb a\; \arb b)$ of $y$, either
$\{a, b\} \subseteq \{1, \ldots, m\}$ or $\{a, b\} \subseteq \{m+1,
\ldots, n\}$. Let $S = \{a \in \{1, \ldots, m\}\; |\;  e_a\cdot y
\in \Phi^-\}$.

We first consider the case where $|S|$ is even. Define $z =
\prod_{a \in S} (\mi a) \in W$. Now
$$N(z) = \{e_a \pm e_b\; |\; 1\leq a < b \leq n, a \in S\}.$$
Since $e_a\cdot y \in \Phi^-$ for every $a \in S$, at least one of
$e_a + e_b$ and $e_a - e_b$ will be in $N(y)$ for all $b > a$.
Therefore $\ell(zy) = \ell(z) + \ell(y) - 2|N(z) \cap N(y)| \leq
\ell(y)$. Similarly $\ell(zx) \leq \ell(z)$. Moreover, by the same
reasoning as that in the proof of Theorem \ref{bnexcess}, $z$
centralizes both $x$ and $y$. Therefore, by Lemma
\ref{afterthought}, we either have a contradiction (forcing $S$ to
be empty and $x, y \in W_J$) or another spartan pair $(zx, zy)$,
this time contained in $W_J$.
Hence $e_J(w) = e(w)$.\\

We are left with the possibility that $|S|$ is odd. If $m=n$, then
since $y \in W(D_n)$ we must have $|S|$ even. So $m < n$. Suppose,
for a contradiction, that $w_2$ consists only of even, positive
cycles. By thinking of $W$ as a subgroup of $W(B_n)$, we can write
$y = y_1y_2$, where $\supp^+(y_1) \subseteq \{1, \ldots, m\}$ and
$\supp^+(y_2) \subseteq \{m+1, \ldots, n\}$. Since $|S|$ is odd,
$y_1$ contains an odd number of sign changes, and since $y \in
W(D_n)$, $y_2$ must also contain an odd number of sign changes.
However $y_2w_2y_2 = w_2^{-1}$. (This is because for every 2-cycle
$(\arb a\; \arb b)$ of $y$, either $\{a, b\} \subseteq \{1, \ldots,
m\}$ or $\{a, b\} \subseteq \{m+1, \ldots, n\}$.) But $w_2$ consists
only of positive, even cycles. This means there are two conjugacy
classes in $D(m+1, \ldots, n)$ with that signed cycle type, but only
one in $B(m+1, \ldots, n)$. Therefore any element of $B(m+1,\ldots,
n)\setminus D(m+1,\ldots, n)$ (such as $y_2$) must interchange the
conjugacy classes. This contradicts the fact that $w_2^{-1}$ is
conjugate in $D(m+1, \ldots, n)$ to $w_2$. Hence if $w_2$ consists
only of even, positive cycles, $|S|$ must be even.

Therefore the only possibility remaining is that $w_2$ contains a
1-cycle $(\+ b)$ or $(\mi b)$. Define $z = (\mi b)\prod_{a \in S}
(\mi a) \in W$. Now
$$N(z) = \{e_a \pm e_c\; |\; 1\leq a < c \leq n, a \in S\} \cup \{e_b \pm e_c\; |\; b < c \leq n\}.$$
Let $\lambda = |\{e_a \pm e_c \; | \; 1 \leq a < c \leq m, a \in
S\}|$. Then $|N(z)| = \lambda + 2|S|(n-m) + 2(n-b)$.
 Now for $c>m$ and $a \in S$, $e_a + e_c \in N(y)$ and $e_a -
e_c \in N(y)$. For $c \leq m$, at least one of $e_a + e_c$, $e_a -
e_c \in N(y)$. Hence \begin{eqnarray*} |N(z) \cap N(y)| &\geq&
\textstyle\frac{1}{2}\lambda + 2|S|(n-m)= \textstyle\frac{1}{2}(\lambda + 4|S|(n-m))\\
&>& \textstyle\frac{1}{2}\left(\lambda + 2|S|(n-m) + 2(n-b)\right)
= \ell(z).\end{eqnarray*} %
Hence $\ell(zy) < \ell(y)$. Similarly $\ell(zx) < \ell(x)$. An
almost identical argument to that in the proof of Theorem
\ref{bnexcess} shows that $z$ centralizes both $x$ and $y$. By Lemma
\ref{afterthought} $(x, y)$ cannot be a spartan pair, a
contradiction. Hence $S$ is empty. This implies that $x, y \in W_J$
and hence $e_J(w) = e(w)$.
\end{proof}

\section{The Set of Roots $N(\I_w)$}

The main result of this section is the proof of Theorem \ref{nwni}.
First we look at the case when $w$ is a cuspidal element. We require
a technical lemma before we can start the proof of Proposition
 \ref{cusplen}.\\

\begin{lemma}\label{tech} Suppose $w \in W$ and let $\Omega_1, \ldots, \Omega_k$ be the
set of $\langle w \rangle$-orbits on $\Phi$. For $x \in
\mathcal{I}_w$, if $\Omega_i\cdot x \cap \Phi^- \neq \emptyset$,
then $\Omega_i \cap \Phi^+ \subseteq N(\mathcal{I}_w)$.
\end{lemma}
\begin{proof} Suppose that $\Omega_i\cdot x
\cap \Phi^- \neq \emptyset$. Then for any $\alpha \in \Omega_i$,
there is an integer $j$ such that $\alpha \cdot w^jx\in \Phi^-$. If
$\alpha \in \Phi^+$, then $\alpha \in N(w^jx)$. Since $w^jx\in
\mathcal{I}_w$, we get $\alpha \in N(\mathcal{I}_w)$. This holds for
each $\alpha \in \Omega_i \cap \Phi^+$, so the lemma is proved.
\end{proof}\\

\begin{prop} \label{cusplen} If $w$ is a cuspidal element of $W$, then $N(\mathcal{I}_w) = \Phi^+$.
\end{prop}

\begin{proof} Assume that $w$ is cuspidal in $W$ and let $\Omega_1, \ldots,
\Omega_k$ be the set of $\langle w \rangle$-orbits on $\Phi$. For $x
\in \mathcal{I}_w$, suppose $\beta \in \Omega_i\cdot x$ and write
$\beta = \alpha\cdot x$. Then for any integer $j$, $\beta \cdot w^j=
\alpha \cdot xw^j= \alpha \cdot w^{-j}x\in \Omega_i\cdot x$.
Therefore $\Omega_i\cdot x$ is also a $\langle w \rangle$-orbit.
Suppose $\Omega = (\beta, \beta\cdot w, \ldots, \beta\cdot w^k)$ is
any $\langle w \rangle$-orbit of roots. Then $v = \sum_{i=0}^k
\beta\cdot w^i$ is a fixed point of $w$. It is well-known (and
follows from Ch V \S 3.3 of \cite{titsref}) that the stabilizer of
any non-zero $v \in V$ is a proper parabolic subgroup of $W$. Since
$w$ is cuspidal, therefore, we must have $v = 0$. Hence $\Omega$
must contain both positive and negative roots. In particular, for
$1\leq i \leq k$, $\Omega_i\cdot x$ must contain at least one
negative root. Therefore, by Lemma \ref{tech}, $\Omega_i \cap \Phi^+
\subseteq N(\mathcal{I}_w)$ for $1\leq i \leq k$. Hence
$N(\mathcal{I}_w) = \Phi^+$, and the result holds.\end{proof}

\begin{lemma} \label{centre} Suppose $W$ has a non-trivial centre. Then for all
$w \in W, N(\mathcal{I}_w) = \Phi^+$. \end{lemma}

\begin{proof} Let $w_0$ be the non-trivial central involution. Let
$x \in \mathcal{I}_w$. Then $xw_0 \in \mathcal{I}_w$. Now %
\begin{eqnarray*} N(xw_0) &=& N(x)\setminus [-N(w_0)\cdot x] \cup
\left[ N(w_0) \setminus N(x)\right]\cdot x\\
&=& N(x) \setminus [\Phi^-\cdot x] \cup \left[\Phi^+\setminus
N(x)\right]\cdot x\\
&=& \emptyset \cup \Phi^+ \setminus N(x) = \Phi^+ \setminus N(x).
\end{eqnarray*}
Hence $N(\mathcal{I}_w) \supseteq N(xw_0) \cup N(x) = \Phi^+$.
Therefore $N(\mathcal{I}_w) = \Phi^+$.\end{proof}\\

\begin{lemma} \label{an} Suppose $W$ is of type $A_{n-1}$. Then for all
$w \in W$, $N(w) \subseteq N(\mathcal{I}_w)$.\end{lemma}
\begin{proof} Now $W$ is isomorphic to Sym($n$), so we may write $w$
as a product of disjoint cycles. Consider a cycle $\lambda =
(a_1a_2\ldots a_m)$ of $w$. Let $\Lambda = \{a_1, \ldots, a_m\}$.
We will define $\sigma_{\lambda}$ and $\sigma'_{\lambda} \in
\sym_{\Lambda} \cap \mathcal{I}_{\lambda}$.  If $m$ is odd, let
$\sigma_{\lambda} = \sigma'_{\lambda} = (a_2 a_m)(a_3 a_{m-1})
\cdots (a_{(m+1)/2} a_{(m+3)/2})$. If $m$ is even, let
$\sigma_{\lambda} = (a_1 a_m)(a_2 a_{m-1})(a_3 a_{m-2}) \cdots
(a_{m/2} a_{m/2 +1})$ and $\sigma'_{\lambda} = (a_2 a_m)(a_3
a_{m-1}) \cdots (a_{m/2} a_{m/2+2})$. Note that every involution
in $\sym_{\Lambda} \cap \mathcal{I}_{\lambda}$ is of the form
$\sigma_{\lambda}$ or $\sigma'_{\lambda}$ for some choice of
$a_1$. Further, note that $\sigma_{\lambda}$ contains the 2-cycle
$(a_{\lceil m/2\rceil} a_{\lceil m/2\rceil+1})$, which is of the
form $(a\; aw)$. \\

Let $e_a-a_b \in N(w)$. Suppose $a$ and $b$ appear in the same cycle $\lambda$ of $w$.
Write $w = \lambda w'$ where $\supp(w') \cap \Lambda = \emptyset$. Let $\tau$ be an
arbitrary element of $\mathcal{I}_{w'} \cap \sym(\supp(w'))$. If $a, b$ are separated by
an even number of $a_i$, then by a judicious choice of $a_1$, we can ensure that $(ab)$
is a 2-cycle of $\sigma_{\lambda}$. Otherwise we can ensure that $(ab)$ is a 2-cycle of
$\sigma'_{\lambda}$. Let $x = \sigma_{\lambda}\tau$ or $\sigma'_{\lambda}\tau$
accordingly.
Then $x \in \mathcal{I}_w$ and $e_a-e_b \in N(x)$.\\
Now suppose that $a$ and $b$ appear in different cycles $\lambda_1$
and $\lambda_2$ of $w$. We may choose $\sigma_{\lambda_1}$ and
$\sigma_{\lambda_2}$ such that $\sigma_{\lambda_1}$ contains the
2-cycle $(a\; aw)$ and $\sigma_{\lambda_2}$ contains the 2-cycle
$(b\; bw)$. Writing $w = \lambda_1\lambda_2 w'$ and choosing any
$\tau \in \mathcal{I}_{w'} \cap \sym(\supp(w'))$ we see that for $x
= \sigma_{\lambda_1}\sigma_{\lambda_2}\tau \in \mathcal{I}_w$, $(e_a
- e_b)\cdot x = (e_a - e_b)\cdot w \in \Phi^-$. Hence $e_a - e_b \in
N(w)$ and therefore $N(w) \subseteq N(\mathcal{I}_w)$.
\end{proof}\\

Note that there are examples in type $A_{n-1}$ where
$N(\mathcal{I}_w) \neq \Phi^+$. For instance, in $W(A_6) \cong
\sym(7)$, for $w = (1234)(567)$ we have $e_4-e_5 \notin
N(\mathcal{I}_w)$.

\begin{prop} \label{dn} Let $W$ be of type $D_n$. For all $w \in W$, $N(w)
\subseteq N(\mathcal{I}_w)$.\end{prop}

We work in the environment of $W(B_n)$, as we will be dividing $w$ into a product of
cycles, some of which may be negative. We write $\mathcal{I}^B_w$ for $\mathcal{I}_w$
when we are considering $w$ as an element of $W(B_n)$, and $\mathcal{I}^D_w =
\mathcal{I}^B_w \cap W(D_n)$. Let $W(B_w)$ be the Coxeter group of type $B$ over
$\supp(w)$ and define $\overline{\mathcal{I}}_w = \mathcal{I}^B_w \cap W(B_w)$. For $w
\in W(B_n)$, we write $\hat w$ for the corresponding element of $\sym(n)$. For example,
if $w = (\overset{+}{1}\; \overset{-}{2})$, then $w$ is an even, negative cycle and $\hat
w = (1 \; 2)$.
\\
For every $w \in W(B_n)$, $\mathcal{I}^B_w \neq \emptyset$. If $w$ is negative, then for
any $\tau \in \mathcal{I}^B_w$, exactly one of $\tau$ and $w\tau$ is negative, so
$\mathcal{I}^B_w$ contains both positive and negative elements. If $w$ is positive, then
$w \in W(D_n)$, and so $\mathcal{I}^D_w$ is non-empty,
whence $\mathcal{I}^B_w$ contains at least one positive element.\\

Before we can prove Proposition \ref{dn} we need the following lemma.

\begin{lemma}\label{technical} Let $g$ be an $m$-cycle of $W(B_n)$.
\begin{trivlist} \item{(i)} Suppose $m$ is odd and write $\hat g = (a_1a_2\cdots a_m)$. Then there
exists $\tau_g \in \overline{\mathcal{I}}_g$ such that $\tau_g$ is
positive and $\hat\tau_g = (a_2 a_m)(a_3 a_{m-1}) \cdots
(a_{(m+1)/2}
a_{(m+3)/2})$.\\
\item{(ii)} Let $a \in \supp(g)$. Then there exist $\tau_{g,a}^+
\in \overline{\mathcal{I}}_g$ and $\tau_{g,a}^- \in
\overline{\mathcal{I}}_g$ such that $(\overset{+}{a})$ is a cycle
of $\tau_{g,a}^+$ and $(\overset{-}{a})$ is a cycle of
$\tau_{g,a}^-$. Moreover, $\tau_{g,a}^+$ is negative if and only
if $m$ is even and $g$ is negative; $g\tau_{g,a}^+$ is negative if
and only if $m$ is odd and $g$ is negative; $\tau_{g,a}^-$ is
positive if and only if $m$ is even and $g$ is
positive.\end{trivlist}\end{lemma}

\begin{proof} We consider the cases $m$ odd and $m$ even
separately.\\

Suppose first that $m$ is odd. Since $\overline{\mathcal{I}}_g$ is
non-empty, there exists $\tau \in \overline{\mathcal{I}}_g$, and
$\hat\tau$ inverts $\hat g$ by conjugation. Hence $\hat\tau$ is of
the form $(a_2 a_m)(a_3 a_{m-1}) \cdots (a_{(m+1)/2} a_{(m+3)/2})$
for some labelling $(a_1 \cdots a_m)$ of $\hat g$. The $\tau$
resulting from any choice of labelling for $g$ are conjugate via an
element of the centralizer of $g$, and hence for any labelling $(a_1
a_2 \cdots a_m)$ of $\hat g$, there exists $\tau \in
\overline{\mathcal{I}}_g$ with $\hat\tau = (a_2 a_m)(a_3 a_{m-1})
\cdots (a_{(m+1)/2} a_{(m+3)/2})$. Let $z$ be the central involution
in $W(B_g)$, so $z = (\overset{-}{a_1})(\overset{-}{a_2})\cdots
(\overset{-}{a_m})$ and $z$ is negative. Exactly one of $\tau$ and
$z \tau$ is positive. Let $\tau_g$ be the positive one. Moreover,
setting $a_1 = a$, define $\tau^+_{g, a} = \tau_g$ and $\tau^-_{g,
a} = z\tau_g$. since $\tau_g$ is a positive involution containing
exactly one 1-cycle, the 1-cycle must be positive. Hence $\tau^+_{g,
a}$ contains $(\overset{+}{a})$ and is positive, and $\tau^-_{g, a}$
contains $(\overset{-}{a})$ and is negative. Finally $g\tau^+_{g,
a}$
is negative if and only if $g$ is negative. This establishes part (i), and (ii) for $m$ odd.\\

Now suppose that $m$ is even. Then $g$ is conjugate either to $h =
(\overset{+}{a_1} \; \overset{+}{a_2} \cdots \overset{+}{a_m})$ or
to $(\overset{-}{a_1} \; \overset{+}{a_2} \cdots \overset{+}{a_m}) =
(\overset{-}{a_1})h$. Let $\sigma =
(\overset{+}{a_1})(\overset{+}{a_{m/2}})(\overset{+}{a_2}
\overset{+}{a_m})(\overset{+}{a_3} \overset{+}{a_{m-1}}) \cdots
(\overset{+}{a_{m/2}} \overset{+}{a_{m/2+2}})$. Then $\sigma \in
\overline{\mathcal{I}}_h$, and hence $(\overset{-}{a_1})\sigma \in
\overline{\mathcal{I}}_{(\overset{-}{a_1})h}$. Therefore
$\overline{\mathcal{I}}_g$ contains an element $\tau$ with two
1-cycles, at least one of which is positive. By choice of labelling
of $g$, we can ensure that $(\overset{+}{a})$ is a 1-cycle of
$\tau$.  In addition, $\tau$ is positive if and only if $g$ is
positive. Let $z$ be the central involution in $W(B_g)$. Note that
because $m$ is even, $z$ is positive. Now set $\tau^+_{g, a} = \tau$
and $\tau^-_{g, a} = z\tau$. Then $\tau^+_{g, a}$ is negative if and
only if $g$ is negative, $g\tau^+_{g, a}$ is always positive, and
$\tau^-_{g, a}$ is positive if and only if $g$ is positive. This
establishes part (ii) for $m$ even, so completing the proof of the
lemma.
\end{proof}\\

We may now give the

\paragraph{Proof of Proposition \ref{dn}}
Suppose $\alpha = \pm e_a \pm e_b \in N(w)$. We must find a
positive $x \in \mathcal{I}^B_w$ such that $\alpha \in N(x)$.
There are two cases to consider, depending on whether $a$ and $b$
are
in the same or distinct cycles of $w$.\\

 We first consider the case
that $\{a, b\} \subseteq \supp(g)$ for some $m$-cycle $g$ of $w$.
Suppose $m$ is odd, then by Lemma \ref{technical} and judicious
choice of $a_1$, there exists $\tau_g \in \overline{\mathcal{I}}_g$
such that $(ab)$ is a 2-cycle of $\hat \tau_g$. Hence $ \alpha \cdot
\tau_g= \pm \alpha$. Therefore either $\alpha \in N(\tau_g)$ or
$\alpha \in N(\tau_gg)$. Replacing $\tau_g$ by $\tau_gg$ if
necessary, we have $\alpha \in N(\tau_g)$ and $\tau_g$ is positive
whenever $g$ is positive. Write $w = gw'$. If $\tau_g$ is positive,
then choose a positive $\sigma \in \overline{\mathcal{I}}_{w'}$. If
$\tau_g$ is negative, then $g$ is negative, forcing $w'$ to be
negative (recall that $w$ is positive), thus we may choose a
negative $\sigma \in \overline{\mathcal{I}}_{w'}$. Finally, let $x =
\tau_g\sigma$. Then $x \in \mathcal{I}^B_w$, $x$ is positive and
$\alpha\cdot x = -\alpha$. Hence $\alpha \in N(x)$. If $m$ is even,
just pick any positive $\sigma \in \mathcal{I}^B_w$. Let $z$ be the
central involution in $W(B_g)$. Since $m$ is even, $z$ is positive.
Then $z\sigma \in \mathcal{I}_w$ and $\alpha \in N(\sigma) \cup
N(z\sigma) \subseteq \mathcal{I}_w$.\\

We must now consider the case where $a$ and $b$ appear in different
cycles of $w$. So assume $a \in \supp(g_1)$, where $g_1$ has length
$m_1$, and $b \in \supp(g_2)$, where $g_2$ has length $m_2$. Write
$w = g_1g_2w'$. Let $\tau = \tau^+_{g_1, a}g_1\tau^+_{g_2,b}g_2$.
Note that $\alpha\cdot \tau = \alpha\cdot w$. If $\tau$ is positive,
then let $\sigma$ be any positive element of
$\overline{\mathcal{I}}_{w'}$. If $\tau$ is negative and $w'$ is
negative, then let $\sigma$ be any negative element of
$\overline{\mathcal{I}}_{w'}$. Set $x = \tau\sigma$. Then $x$ is a
positive element of $\mathcal{I}^B_w$, and $\alpha \in N(x)$. The
only case not covered is where $\tau$ is negative and $w'$ is
positive. Hence $g_1g_2$ is positive. Without loss of generality, we
must have $\tau^+_{g_1, a}g_1$ negative and $\tau^+_{g_2, b}g_2$
positive, which means $m_1$ is odd, and $g_1$ is negative; hence
$m_2$ is even and $g_2$ is negative. In this case we let $\sigma$ be
any positive element of $\overline{\mathcal{I}}_{w'}$ and set $x =
\tau^-_{g_1, a}\tau^-_{g_2, b}\sigma$. We see that $x$ is a positive
element of $\mathcal{I}^B_w$. Furthermore $\alpha\cdot x = -\alpha$
and hence $\alpha \in N(x)$. For each possibility we have found a
positive $x \in \mathcal{I}^B_w$ such that $\alpha \in N(x)$, so the
result holds. \qed

\begin{prop} \label{simplylaced} Suppose $W$ is finite with trivial centre. Then for
all $w \in W, N(w) \subseteq N(\mathcal{I}_w)$.
\end{prop}

\begin{proof} It is enough to consider the finite irreducible Coxeter groups $W$
with trivial centre. These are precisely the Coxeter groups of
type $A_n$, $n\geq 1$, $D_n$, $n > 4$ and $n$ odd, $E_6$ and
$\dih(2m)$ for $m$ odd. In $\dih(2m)$ the non-trivial conjugacy
classes are either classes of reflections or cuspidal classes. In
either case the result is trivially true. For $E_6$ the result was
checked with a computer using {\sc Magma}\cite{magma}. Types $A_n$
and $D_n$ have been dealt with in
Lemma \ref{an} and Proposition \ref{dn} respectively.\end{proof}\\

Theorem \ref{nwni} is an immediate consequence of Lemma \ref{centre}
and Proposition \ref{simplylaced}.\\

Theorem \ref{nwni} raises the question as to whether, for $w \in W$,
$N(w) \subseteq N(x)$ for some $x \in \mathcal{I}_w$. However we do
not have to look very far before alighting upon the following
example. Choose $W$ to be the Coxeter group $W(A_4)\cong \sym(5)$
and let $w = (235)$. Now $N(w) = \{e_2 - e_5, e_3 - e_4, e_3 - e_5,
e_4 - e_5\}$. Also $\mathcal{I}_w = \{(23), (35), (25), (14)(23),
(14)(35), (14)(25)\}$ and we have
\begin{center}\begin{tabular}{ll}
$x \in \mathcal{I}_w$ & $N(x)$ \\
\hline %
$(23)$& $\{ e_2-e_3\}$\\
$(35)$& $\{ e_3-e_4, e_3-e_5, e_4-e_5\}$\\
$(25)$& $\{ e_2-e_3, e_2-e_4, e_2-e_5, e_3-e_5, e_4-e_5\}$\\
$(14)(23)$& $\{ e_1-e_2, e_1-e_3, e_1-e_4, e_2-e_3, e_2-e_4, e_3-e_4\}$\\
$(14)(35)$& $\{ e_1-e_2, e_1-e_3, e_1-e_4, e_2-e_4, e_3-e_4, e_3-e_5\}$\\
$(14)(25)$& $\{ e_1-e_3, e_1-e_4, e_1-e_5, e_2-e_3, e_2-e_4,
e_2-e_5, e_3-e_4, e_3-e_5\}$.
\end{tabular}\end{center}%
From the above we observe that for each $x \in \mathcal{I}_w$,
$N(w) \not\subseteq N(x)$.

\end{document}